\documentclass[12pt,reqno]{amsart}
\usepackage{geometry}                
\usepackage{amsmath,amssymb}
\usepackage{graphicx}
\usepackage{amssymb,latexsym, amsmath}
\usepackage{amsfonts,amsthm}
\usepackage{amscd}
\usepackage{mathrsfs}
\usepackage{hyperref}
\usepackage{eucal}
\usepackage{epstopdf}
\usepackage{amsgen}
\usepackage{xspace}
\usepackage{verbatim}
\usepackage{stmaryrd}
\usepackage{enumitem}
\newlist{steps}{enumerate}{1}
\setlist[steps, 1]{label = Step \arabic*:}

\newcommand{\gd}{\delta}

\newcommand{\inpt}[1]{\langle #1 \rangle}

\newcommand{\ga}{\gamma}

\newcommand{\G}{\Gamma}
\newcommand{\gl}{\lambda}

\newcommand{\pdr}{\partial}

\newcommand{\beq}{\begin{equation}}
\newcommand{\eeq}{\end{equation}}
\newcommand{\bea}{\begin{align}}
\newcommand{\eea}{\end{align}}
\newcommand{\bthm}{\begin{theorem}}
\newcommand{\ethm}{\end{theorem}}
\newcommand{\bpr}{\begin{proof}}
\newcommand{\epr}{\end{proof}}
\newcommand{\bcl}{\begin{corollary}}
\newcommand{\ecl}{\end{corollary}}
\newcommand{\bpn}{\begin{proposition}}
\newcommand{\epn}{\end{proposition}}
\newcommand{\bre}{\begin{remark}}
\newcommand{\ere}{\end{remark}}
\newcommand{\bdf}{\begin{definition}}
\newcommand{\edf}{\end{definition}}
\newcommand{\bss}{\begin{align*}}
\newcommand{\ess}{\end{align*}}

\newcommand{\bl}{\label}

\newcommand{\mR}{\mathbb{R}}

\newcommand{\mZ}{\mathbb{Z}}

\newcommand{\ik}{{i k}}
\newcommand{\smn}{\sum_{i=1}^m \sum_{k=1}^n}
\newcommand{\zmn}{\mathbb{Z}_{mn}}

\newtheorem{theorem}{Theorem}[section]
\newtheorem{corollary}[theorem]{Corollary}

\newtheorem{proposition}[theorem]{Proposition}

\theoremstyle{definition}
\newtheorem{definition}[theorem]{Definition}
\theoremstyle{remark}
\newtheorem{remark}{Remark}

\numberwithin{equation}{section}

\begin{document}

\title[Synchronization of 2D Cellular Neural Networks]{Exponential Synchronization of 2D Cellular Neural Networks with Boundary Feedback}

\author[L. Skrzypek]{Leslaw Skrzypek}
\address{L. Skrzypek, Department of Mathematics and Statistics, University of South Florida, Tampa, FL 33620, USA}
\email{skrzypek@usf.edu}
\thanks{}

\author[C. Phan]{Chi Phan}
\address{C. Phan, Department of Mathematics and Statistics, Sam Houston State University, Huntsville, TX 77340, USA}
\email{chp007@shsu.edu}
\thanks{}

\author[Y. You]{Yuncheng You}
\address{Y. You (Emeritus),  Department of Mathematics and Statistics, University of South Florida, Tampa, FL 33620, USA}
\email{you@mail.usf.edu}
\thanks{}


\date{}



\begin{abstract} 
In this work we propose a new model of 2D cellular neural networks (CNN) in terms of the lattice FitzHugh-Nagumo equations with boundary feedback and prove a threshold condition for the exponential synchronization of the entire neural network through the \emph{a priori} uniform estimates of solutions and the analysis of dissipative dynamics. The threshold to be satisfied by the gap signals between pairwise boundary cells of the network is expressed by the structural parameters and adjustable. The new result and method of this paper can also be generalized to 3D and higher dimensional FitzHugh-Nagumo type or Hindmarsh-Rose type cellular neural networks.
\end{abstract}

\maketitle

\section{\textbf{Introduction}}

For the collective dynamic phenomena of many network systems that attract scientific research interests, two of the most ubiquitous concepts for the relevant mathematical models presented by various differential equations are synchronization and pattern formation \cite{AA, A, Chow2, Chua2, I, PC, PRK}. The mechanisms of synchronization and control depend on the spatiotemporal structures of the designed models. 

In this work we shall focus on the two-dimensional cellular neural networks (briefly CNN) modeled by the 2D lattice FitzHugh-Nagumo equations with the boundary feedback and prove the exponential synchronization when the adjustable and explicit threshold condition is satisfied by the pairwise boundary gap signals. 

As well known, CNN was invented by Chua and Yang \cite{Chua1, Chua2} in 1988. CNN physically consists of network-like interacted analog or digital signal processors such as VLSI. Mathematically, CNN is defined \cite{Chua3, CTR} to be a 2D, 3D or higher dimensional array of identical template dynamical systems (called cells), which satisfies two properties that the interactions (called the synaptic laws) are local within a neighborhood of finite radius $r > 0$ and the state variables are all time-continuous signals. In a large sense the dynamics of CNN can be studied as a lattice dynamical system generated by lattice differential equations in time \cite{Chow1, Chow2, Chua3, CR, S2}. 

The diversified cellular neural networks as well as the Hopfield neural networks have found effective applications in many areas such as computational image processing, medical visualization, data driven optimization, pattern recognition, associative memory, and secure communications, cf. \cite{Chow2, Chua2, CR, GJH, L, PC, SG, S1, S2, Wu}. 

In the expanding front of deep learning and artificial intelligence in general, the theory of cellular neural networks, convolutional neural networks, and variants of complex neural networks is closely linked to discrete nonlinear partial differential equations and delay differential equations \cite{Chow1, CR, GJH, S2}. 

This work aims to prove the exponential synchronization of the 2D FitzHugh-Nagumo cellular neural networks with the new feature of the boundary feedback, which is a substantial generalization of feedback synchronization for one-dimensional FitzHugh-Nagumo cellular neural networks \cite{LY2} shown by the authors. In this CNN model, the cell template with the synapsis is described by the discrete version of the 2D partly diffusive FitzHugh-Nagumo equations with boundary feedback control.

We consider a cellular neural network of the 2D grid-structure, which consists of the cells $\{N(i, k)\}$ located at the grid points $\{(ih_x, kh_y): i = 1, 2, \cdots , m \;\text{and} \; k = 1, 2, \cdots, n\}$ for given $h_x, h_y > 0$ along the $x$-row direction and the $y$-column direction, respectively. We shall study synchronization problem of the following 2D lattice FitzHugh-Nagumo equations: 
\beq \bl{CN}	
	\begin{split}
	\frac{dx_\ik}{dt} &= a [(x_{i-1, k} - 2x_\ik + x_{i+1,k}) + (x_{i, k-1} - 2x_\ik + x_{i, k+1})]  \\
	&\quad + f(x_\ik) - b\, y_\ik + p u_\ik, \\
	\frac{dy_\ik}{dt} &= c\, x_\ik - \gd \, y_\ik, 
	\end{split}
\eeq
where $1 \leq i \leq m, \; 1 \leq k \leq n, \; t > 0$, the integers $m, n \geq 4$, and the two-dimensional discrete Laplacian operator \cite{Chow1, Chua1, CR, S2}
\begin{align*}
	D_\ik (x) &= a  [(x_{i-1, k} - 2x_\ik + x_{i+1,k}) + (x_{i, k-1} - 2x_\ik + x_{i, k+1})]   \\
	&= a (u_{i-1,k} + u_{i+1, k} + u_{i, k-1} + u_{i, k+1} - 4 u_\ik)
\end{align*}
shows the synaptic law of cell coupling. The nonlinear function $f(\cdot )$ will be specified below. In this 2D model of CNN, we consider the periodic boundary condition:
\beq \bl{pbc}
	\begin{split}
	x_{0, k} (t) &= x_{m, k} (t), \quad x_{m+1, k}(t) = x_{1, k} (t), \quad \text{for} \;\, 1 \leq k \leq n,   \\
	x_{i,  0} (t) &= x_{i, n} (t), \, \; \quad x_{i, n+1}(t) = x_{i, 1} (t), \;\, \quad \text{for} \;\, 1 \leq i \leq m, 
	\end{split}
\eeq
and the \emph{boundary feedback} control $\{u_{i\, k}: 1 \leq i \leq m, 1 \leq k \leq n\}$: For $t \geq 0$,
\beq \bl{bfc1}
	\begin{split}
	&u_{1, k} (t) = u_{m+1, k} (t) = x_{m, k} (t) - x_{1, k}(t),   \quad \text{for} \; 1 \leq k \leq n,   \\
	&u_{i, k} (t) = 0, \quad 2 \leq i \leq m - 1,  \quad \text{for} \;\; 1 \leq k \leq n,\\
	&u_{m, k} (t) = u_{0, k} (t) = x_{1, k} (t) - x_{m, k} (t),  \quad \text{for} \;\; 1 \leq k \leq n,
	\end{split}
\eeq
and 
\beq \bl{bfc2}
	\begin{split}
	&u_{i, 1} (t) = u_{i, n+1} (t) = x_{i, n} (t) - x_{i, 1}(t),  \quad  \text{for} \; 1 \leq i \leq m,   \\
	&u_{i, k} (t) = 0, \quad 2 \leq i \leq n - 1,  \quad \text{for} \;\; 1 \leq i \leq m,\\
	&u_{i, n} (t) = u_{i, 0} (t) = x_{i, 1} (t) - x_{i, n} (t), \quad \text{for} \;\; 1 \leq i \leq m.
	\end{split}
\eeq
All the parameters $a, b, c, \gd, p$ can be any given positive constants, and $p > 0$ is the adjustable coefficient of the boundary feedback signals. Here $x_{m, k} (t) - x_{1, k}(t)$ measures the boundary gap signal between the two boundary node cells on the same row of the cellular neural network and $x_{i, n} (t) - x_{i, 1} (t)$ measures the boundary gap signal between the two boundary node cells on the same column of the cellular neural network. The initial conditions for the system \eqref{CN} are denoted by
\beq \bl{inc}
	x_\ik (0) = x_\ik^0 \in \mR \quad \text{and} \quad y_\ik (0) = y_\ik^0 \in \mR, \quad 1 \leq i \leq m, \;\; 1 \leq k \leq n.
\eeq

We make the following Assumption: The scalar function $f \in C^1 (\mR, \mR)$ satisfies 
\beq \bl{Asp}
	\begin{split}
	&f(s) s \leq  - \gl s^4 + \beta, \quad s \in \mathbb{R}, \\
	& f^{\,\prime} (s) \leq \ga, \quad s \in \mathbb{R}, \\
	\end{split}
\eeq
where $\gl, \beta$ and $\ga$ can be any given positive constants. Note that the nonlinear term in the original FitzHugh-Nagumo ordinary differential equations \cite{FH} is
$$
	f(s) = s(s-\alpha)(1 - s)
$$ 
and the constant $0 < \alpha < 1$. It satisfies the Assumption \eqref{Asp}:
\begin{align*}
	f(s)s &= - \alpha s^2 + (\alpha + 1)s^3- s^4 \leq -\alpha s^2 + \left(\frac{1}{2}s^4 + 2^3 (\alpha +1)^4\right) - s^4 \\
	&\leq - \left(\alpha s^2 + \frac{1}{2}s^4 \right) + 8(\alpha + 1)^4 \leq - \frac{1}{2} s^4 + 8(\alpha + 1)^4,  \\[3pt]
        f^{\, \prime} (s) &= - \alpha + 2(\alpha +1)s - 3s^2 \leq - \alpha + (\alpha +1)^2 - 2s^2 \leq 1 + \alpha + \alpha^2.
\end{align*}

Synchronization and its control play a significant role for biological neural networks and for the artificial neural networks as well \cite{A, I, WC}. Fast and effective synchronization may lead to enhanced functionality and performance of complex neural networks. 

For biological or artificial neural networks, synchronization topics have been studied with several mathematical models, including the FitzHuigh-Nagumo neural networks, typically with the synaptic coupling by clamped gap junctions \cite{AA, A, IJ, Yong} and the mean field couplings \cite{QT, WLZ}. For chaotic and stochastic neural networks with various applications, the pinning control is usually exploited \cite{L, PC, SG, WZML, ZCW}. Exponential synchronization of neural networks with or without time delays has also been studied in \cite{BW, CLH, GZG, YHJ}.

The methods commonly used in the reported researches on stability and synchronization of CNN and complex neural networks are mainly based on the analysis of eigenvalues for the coupled matrices, the linear matrix inequalities \cite{CR, FBA, GM, LPK, PC}, and the Lyapunov functionals \cite{CLH, I, S1}, with many references therein.                                                          

Recently the authors proved results on the exponential synchronization for the boundary coupled Hindmarsh-Rose neuron networks in \cite{PLY, PY}, the boundary coupled partly diffusive FitzHugh-Nagumo neural networks in \cite{LY1}, and the feedback synchronization of the one-dimensional FitzHugh-Nagumo CNN in \cite{LY2}.

The feature of this work is to present a new model of 2D FitzHugh-Nagumo CNN with the computationally favorable boundary feedback and to prove a sufficient condition for realization of its exponential synchronization. Moreover, this work is characterized by a new mathematical approach of dynamical \emph{a priori} estimates to show the existence of absorbing set for the solution semiflow of this CNN, which leads to the main result on the threshold condition for the exponential synchronization. The threshold is explicitly expressed in terms of the neural network parameters and can be adjusted by the strength coefficient $p$ of the boundary feedback in applications. 

\section{\textbf{Absorbing Set and Dissipative Dynamics}}

Define the following Hilbert space:
$$
	H = \ell^2 (\mZ_{mn}, \mR^{2mn}) = \{ (x, y) = ((x_\ik, y_\ik): 1 \leq i \leq m, 1 \leq k \leq n) \} 
$$
where $\mZ_{m n} = \{1 ,2, \cdots, m\} \times \{1, 2, \cdots , n\}$ and $m, n \geq 4$. The norm in $H$ is denoted and define by $\|(x, y)\|^2 =  \sum_{i=1}^m \sum_{k=1}^n (| x_\ik |^2 + |y_\ik |^2)$. The inner-product of $H$ or $\mathbb{R}^d$ is denoted by $\inpt{\,\cdot , \cdot\,}$. The space of all continuous and bounded functions of time $t \geq 0$ valued in $H$ is denoted by $C([0, \infty), H)$, which is a Banach space with the sup-norm. 

Since there exists a unique local solution in time of the initial value problem \eqref{CN}-\eqref{inc} under the Assumption \eqref{Asp} because the right-hand side functions in \eqref{CN} are locally Lipschitz continuous, we shall first prove the global existence in time of all the solutions in the space $H$. Then by the uniform estimates we show the dissipative dynamics of the solution semiflow in terms of the existence of an absorbing set.

\begin{theorem} \label{Tm}
Under the setting in Section \textup{1}, for any given initial state $((x_\ik^0, y_\ik^0): (i, k) \in \zmn) \in H$, there exists a unique solution 
$$
	((x_\ik (t, x_\ik^0), y_\ik (t, y_\ik^0)): (i, k) \in \zmn, \, t \geq 0 ) \in C([0, \infty), H) 
$$ 
of the initial-boundary value problem \eqref{CN}-\eqref{inc} for this 2D FitzHugh-Nagumo cellular neural network. 
\end{theorem}

\begin{proof}
Multiply the $x_\ik$-equation in \eqref{CN} by $C_1 x_\ik (t)$ for $(i,k) \in \zmn$, where the constant $C_1 > 0$ is to be chosen, then sum them up and by the Assumption \eqref{Asp} to get	

\begin{equation} \bl{u1}
	\begin{split}
	&\frac{C_1}{2} \frac{d}{dt} \sum_{i = 1}^m \sum_{k=1}^n |x_\ik |^2 = C_1 \smn \left[\, a(x_{i-1, \, k} - 2x_\ik + x_{i+1, \, k}) x_\ik  \right. \\[4pt]
	&\left. + a(x_{i, \, k-1} - 2x_\ik + x_{i, \, k+1}) x_\ik + f(x_\ik) x_\ik - b x_\ik  y_\ik + p u_\ik x_\ik \, \right]   \\[4pt]
	\leq &\, C_1 \smn \left[\, a(x_{i-1, \, k} - 2x_\ik + x_{i+1, \, k}) x_\ik  + a(x_{i, \, k-1} - 2x_\ik + x_{i, \, k+1}) x_\ik \, \right] \\
	+ &\, C_1 \smn \left[- \gl |x_\ik |^4 + \beta + \frac{b}{2}\, |x_\ik |^2 + \frac{b}{2}\, |y_\ik |^2 \right]  \\
	- &\, C_1 \smn p \left[(x_{1, k} - x_{m, k})^2 + (x_{i,1} - x_{i, n})^2 \right], 
	\end{split}
\end{equation}
for $t \in I_{max} = [0, T_{max})$, which is the maximal existence interval of the solution. By the discrete 'divergence' formula and the boundary condition \eqref{pbc}, we have
\beq \bl{key}
	\begin{split}
	&\smn \left[\, a(x_{i-1, \, k} - 2x_\ik + x_{i+1, \, k}) x_\ik + a(x_{i, \, k-1} - 2x_\ik + x_{i, \, k+1}) x_\ik \right]  \\
	= &\, \smn a [ (x_{i+1, k} - x_\ik)x_\ik - (x_\ik - x_{i-1, k})x_\ik ]  \\
	&\,+ \smn a [ (x_{, k+1} - x_\ik)x_\ik - (x_\ik - x_{i, k-1})x_\ik ]  \\
	= &\, \sum_{k=1}^n \left[ \sum_{i=1}^{m-1} a(x_{i+1,k} - x_\ik) x_\ik- \sum_{i=2}^m a(x_\ik - x_{i-1, k})x_\ik \right]  \\
	&\, + \sum_{k=1}^n \left[ a(x_{m+1, \,k} - x_{m\, k})x_{m\, k} - a(x_{1, \,k} - x_{0, \,k})x_{1,\, k} \right]  \\
	&\, + \sum_{i=1}^m \left[ \sum_{k=1}^{n-1} a(x_{i,\, k+1} - x_\ik) x_\ik- \sum_{k=2}^n a(x_\ik - x_{i,\, k-})x_\ik \right]  \\
	&\, + \sum_{i=1}^m \left[ a(x_{i, \,n+1} - x_{i\, n})x_{i\, n} - a(x_{i, \,1} - x_{i, \,0})x_{i,\, 1} \right]  \\
	= &\, - \sum_{k=1}^n \sum_{i=1}^{m} a(x_\ik - x_{i-1, \, k})^2 - \sum_{i=1}^m \sum_{k=1}^{n} a(x_\ik - x_{i, \, k-1})^2\leq 0.
	\end{split}
\eeq
Then \eqref{u1} with \eqref{key} yields the differential inequality
\beq \bl{u2}
	\begin{split}
	&C_1 \frac{d}{dt} \smn |x_\ik |^2 + 2C_1 p\, \smn \left[(x_{1, k} - x_{m, k})^2 + (x_{i, 1} - x_{i, n})^2 \right] \\
	\leq &\,C_1 \smn \left[- 2\gl |x_\ik (t)|^4 + 2\beta + b |x_\ik (t)|^2 + b |y_\ik (t)|^2 \right],  \quad t \in I_{max}.
	\end{split}
\eeq

Next multiply the $y_\ik$-equation in \eqref{CN} by $y_\ik (t)$ for $1 \leq i \leq m, 1 \leq k \leq n$ and then sum them up. By using Young's inequality, we obtain
\begin{equation} \bl{w1}
	\begin{split}
	&\frac{1}{2} \frac{d}{dt} \smn |y_\ik (t) |^2 = \smn ( cx_\ik \, y_\ik - \gd y_\ik^2) \\
	\leq &\, \smn \left[\left(\frac{c^2}{\gd} x_\ik^2 + \frac{1}{4} \gd \,y_\ik^2\right) - \gd \,y_\ik^2\right]  \\
	= &\, \smn \left[\frac{c^2}{\gd} \, |x_\ik (t)|^2 - \frac{3}{4} \gd \,|y_\ik(t)|^2\right], \quad \text{for} \;\, t \in I_{max}.
	\end{split}
\end{equation}
Add up the inequalities \eqref{u2} and doubled \eqref{w1}. We obtain
\beq \bl{uw}
	\begin{split}
        &\frac{d}{dt} \smn \left[C_1 |x_\ik (t)|^2 + | y_\ik (t)|^2 \right] \\
        &+ 2C_1 p\, \smn \left[(x_{1, k} - x_{m, k})^2 + (x_{i, 1} - x_{i, n})^2 \right] \\
        \leq &\smn \left[ \left(C_1 b + \frac{2c^2}{\gd}\right) |x_\ik (t)|^2 - 2C_1 \gl |x_\ik (t)|^4 + 2C_1 \beta \right] \\
        & + \smn \left[\left(C_1 b - \frac{3 \gd}{2}\right) |y_\ik (t)|^2 \right], \;\;  t \in I_{max} = [0, T_{max}).
	\end{split}
\eeq
We now choose the constant 
\beq \bl{C1}
	C_1 = \frac{\gd}{2b} \quad \text{so that} \quad C_1 b - \frac{3\gd}{2} = - \gd.
\eeq
Then from \eqref{uw} with the fact $2C_1 p \smn [ \cdots ] \geq 0$ on the left-hand side and from the choice of the constant $C_1$ in \eqref{C1}, we have

\begin{gather*}
	 \frac{d}{dt} \smn \left(C_1 |x_\ik |^2 + | y_\ik |^2 \right)   \\
	 \leq \smn \left[ \left(C_1 b + \frac{2c^2}{\gd}\right)|x_\ik (t)|^2- 2C_1 (\gl |x_\ik (t)|^4 + \beta) - \gd |y_\ik (t)|^2\right] 
\end{gather*}
and consequently,
\beq \bl{Cuw}
	\begin{split}
	&\frac{d}{dt} \smn \left(C_1 |x_\ik (t) |^2 + | y_\ik (t)|^2 \right) + \gd  \smn \left( C_1 |x_\ik (t) |^2 + | y_\ik (t)|^2 \right)   \\
	\leq &\, \smn \left[ \left(C_1 b + C_1 \gd + \frac{2c^2}{\gd}\right)|x_\ik (t)|^2- 2C_1 (\gl |x_\ik (t)|^4 + \beta)\right]   \\
	= &\, \smn  \left[ \left(\frac{\gd}{2} + \frac{\gd^2}{2b} + \frac{2c^2}{\gd}\right) |x_\ik (t)|^2- \frac{\gd \gl}{b} |x_\ik (t)|^4 + \frac{\gd \beta}{b} \right], \quad t \in I_{max}.
	\end{split}
\eeq
Completing square shows that
\begin{align*}
	 &\left(\frac{\gd}{2} + \frac{\gd^2}{2b} + \frac{2c^2}{\gd} \right) |x_\ik (t)|^2- \frac{\gd \gl}{b} |x_\ik (t)|^4  \\
	 = &\, - \frac{\gd \gl}{b} \left[ | x_\ik (t)|^2 - \frac{b}{2\gd \gl} \left(\frac{\gd^2}{2b} + \frac{\gd}{2} + \frac{2 c^2}{\gd}\right) \right]^2 + C_2
\end{align*}
and
\beq \bl{C2}
	C_2 = \frac{b}{4\gd \gl} \left(\frac{\gd^2}{2b} + \frac{\gd}{2} + \frac{2 c^2}{\gd}\right)^2.
\eeq
Therefore, \eqref{Cuw} yields
\beq \bl{Suw}
	\frac{d}{dt} \smn (C_1 |x_\ik |^2 + | y_\ik |^2) + \gd  \smn ( C_1 |x_\ik |^2 + | y_\ik |^2) \leq mn \left[C_2 + \frac{\gd \beta}{b}\right], \;\, t \in I_{max}.
\eeq
Apply the Gronwall inequality to \eqref{Suw}. Then we have the following bounded estimate for all the solutions of the system \eqref{CN}-\eqref{inc},
\beq \label{dse}
	\begin{split}
	& \smn \left(|x_\ik (t, x_\ik^0) |^2 + |y_\ik (t, y_\ik^0)|^2 \right) \\
	\leq &\, \frac{1}{\min \{C_1, 1\}} \left[e^{- \gd \, t} \smn (C_1 |x_\ik^0 |^2 + |y_\ik^0 |^2) + \frac{mn}{\gd} \left(C_2 + \frac{\gd \beta}{b}\right)\right], \;\; t \in [0, \infty). \\
	\end{split}
\eeq
Here we can assert that $I_{max} = [0, \infty)$ for all the solutions because the bounded estimate \eqref{dse} shows that the solutions will never blow up at any finite time.  Thus it is proved that for any given initial state in $H$ there exists a unique global solution $((x_\ik (t, x_\ik^0), y_\ik (t, y_\ik^0)): (i,k) \in \zmn ), \, t \in [0, \infty)$, in $H$. 
\end{proof}

The global existence and uniqueness of the solutions to the initial-boundary value problem \eqref{CN}-\eqref{inc} and their continuous dependence on the initial data enable us to define the solution semiflow $\{S(t): H \to H\}_{t \geq 0}$ of this system of the two-dimensional FitzHugh-Nagumo cellular neural network:
$$
	S(t): ((x_\ik^0, y_\ik^0): (i,k) \in \zmn) \longmapsto ((x_\ik (t, x_\ik^0), y_\ik (t, y_\ik^0)): (i,k) \in \zmn).
$$
We call $\{S(t)\}_{t \geq 0}$ the semiflow of the FitzHugh-Nagumo CNN with the boundary feedback.

\begin{theorem} \label{Dsp}
	The semiflow $\{S(t)\}_{t \geq 0}$ of the FitzHugh-Nagumo CNN with the boundary feedback is dissipative in the sense that there exists a bounded ball in the space $H$,
\beq \label{abs}
	B^* = \{g \in H: \| g \|^2 \leq Q\}
\eeq 
where the constant, which is independent of any initial data,
\beq \bl{Q}
	Q = \frac{1}{\min \{C_1, 1\}} \left[1 +  \frac{mn}{\gd} \left(C_2 + \frac{\gd \beta}{b}\right) \right]
\eeq
such that for any given bounded set $B \subset H$, there is a finite time $T_B > 0$ and all the solutions with the initial state inside the set $B$ will permanently enter the ball $B^*$ for $t \geq T_B$. Then $B^*$ is called an absorbing set of the semiflow in the space $H$.
\end{theorem}

\begin{proof}
The bounded estimate \eqref{dse} implies that
	\beq \label{lsp}
	\limsup_{t \to \infty} \,\smn \left(|x_\ik (t, x_\ik^0) |^2 + |y_\ik (t, y_\ik^0)|^2 \right) < Q 
	\eeq
for all the solutions of \eqref{CN}-\eqref{inc} with any initial data $((x_\ik^0, y_\ik^0): (i,k) \in \zmn) \in H$. Moreover, for any given bounded set $B = \{g \in H: \|g \|^2 \leq \rho \}$ in $H$, there is a finite time 
$$
	T_B = \frac{1}{\gd} \log^+ (\rho \, \max \{C_1, 1\} )
$$	
such that 
$$
	e^{- \gd \, t} \smn \left(C_1 |x_\ik^0 |^2 + |y_\ik^0 |^2 \right) < 1, \quad \text{for} \;\, t \geq T_B,
$$
which means all the solution trajectories started from the set $B$ will uniformly and permanently enter the bounded ball $B^*$ shown in \eqref{abs} for $t \geq T_B$. Therefore, the ball $B^*$ in \eqref{abs} is an absorbing set and this semiflow of the FitzHugh-Nagumo CNN with the boundary feedback is dissipative in $H$.
\end{proof}

\section{\textbf{Synchronization of the 2D FitzHugh-Nagumo CNN}} 

Define the difference of solutions for two adjacent double-indexed cells of this FitzHugh-Nagumo CNN \eqref{CN} to be
\beq \bl{DF}
	\begin{split}
	\G_\ik (t) &= x_\ik (t) - x_{i-1,k} (t), \quad V_\ik (t) = y_\ik (t) - y_{i-1,k} (t), \;\; \text{for}\;\, (i, k) \in \zmn ; \\
	\Pi_\ik (t) &= x_\ik (t) - x_{i,k-1} (t), \quad W_\ik (t) = y_\ik (t) - y_{i,k-1} (t), \;\; \text{for}\;\, (i, k) \in \zmn .
	\end{split}
\eeq
We shall consider the system of the \emph{row-differencing} equations for this CNN:
\beq \bl{dHR}
	\begin{split}
		\frac{\pdr \G_\ik}{\pdr t} = \,& a (\G_{i-1,k} - 2\G_\ik + \G_{i+1,k}) + a (\G_{i, \,k-1} - 2\G_\ik + \G_{i, \,k+1})  \\[2pt]
		& + f(x_\ik) - f(x_{i-1,k}) - b V_\ik + p(u_\ik - u_{i-1,k}),  \\
		\frac{\pdr V_\ik }{\pdr t} = \, & c\, \G_\ik - \gd V_\ik,    \quad \text{for} \;\, 1 \leq k \leq n;
	\end{split}
\eeq
and the system of the \emph{column-differencing} equations for this CNN:
\beq \bl{gHR}
	\begin{split}
		\frac{\pdr \Pi_\ik}{\pdr t} = \,& a (\Pi_{i-1,k} - 2\Pi_\ik + \Pi_{i+1,k}) + a (\Pi_{i, \,k-1} - 2\Pi_\ik + \Pi_{i, \,k+1})  \\[2pt]
		& + f(x_\ik) - f(x_{i, k-1}) - b W_\ik + p(u_\ik - u_{i, k-1}),  \\
		\frac{\pdr W_\ik }{\pdr t} = \, & c\, \Pi_\ik - \gd W_\ik,   \quad \text{for} \;\, 1 \leq i \leq m.
	\end{split}
\eeq
According to the periodic boundary condition \eqref{pbc}, the corresponding boundary condition for the equations \eqref{dHR} is 
\beq \bl{pBC}
	\begin{split}
	\G_{0, k} (t) & = \G_{m,k} (t), \quad \G_{m+1, k} (t) = \G_{1, k} (t), \quad \text{for} \;\;  1 \leq k \leq n;  \\
	\Pi_{i, 0} (t) & = \Pi_{i, n} (t), \, \quad \Pi_{i, n+1} (t) = \Pi_{i, 1} (t), \; \quad \text{for} \;\;  1 \leq i \leq m.
	\end{split}
\eeq

Here is the main result on the synchronization of this 2D FitzHugh-Nagumo CNN with the boundary feedback.

\begin{theorem} \bl{ThM}
	If the following threshold condition for the boundary gap signals of the 2D FitzHugh-Nagumo cellular neural network \eqref{CN}-\eqref{inc} is satisfied,  
\beq \bl{SC}
	 \liminf_{t \to \infty} \left[\sum_{k=1}^n |x_{m,k}(t) - x_{1,k}(t)|^2 + \sum_{i=1}^m |x_{i, n}(t) - x_{i, 1}(t)|^2 \right] > Q \left[1 + \frac{1}{p} (\gd + \ga + 2|c - b|) \right]
\eeq
where the constant $Q$ is given in \eqref{Q}, then this cellular neural network is asymptotically synchronized in the space $H$ at a uniform exponential rate. That is, for any initial data $((x_\ik^0, y_\ik^0): (i,k) \in \zmn) \in H$,
\beq \bl{rsyn}
	\begin{split}
	&\lim_{t \to \infty} \smn \left(| x_\ik (t) - x_{i-1, k}(t)|^2 + |y_\ik (t) - y_{i-1, k}(t)|^2 \right)   \\
	+ &\, \lim_{t \to \infty} \smn \left(|x_\ik (t) - x_{i, k-1} (t)|^2 + |y_\ik (t) - y_{i, k-1} (t)|^2 \right)   \\
	= &\, \lim_{t \to \infty} \smn \left(|(\G_\ik (t)|^2 + |V_\ik (t)|^2 + |\Pi_\ik (t)|^2 + |W_\ik (t)|^2\right) = 0,
	\end{split}
\eeq
\end{theorem}

\begin{proof}
	 For any given $1 \leq k \leq n$, multiply the first equation in \eqref{dHR} by $\G_\ik (t)$ and the second equation in \eqref{dHR} by $V_\ik (t)$. For any given $1 \leq i \leq m$, multiply the first equation in \eqref{gHR} by $\Pi_\ik (t)$ and the second equation in \eqref{gHR} by $W_\ik (t)$. Then sum all of them up for all $(i,k) \in \zmn $. We obtain
\beq \bl{eG}
	\begin{split}
	&\frac{1}{2} \frac{d}{dt} \smn \left(|\G_\ik (t)|^2 + |V_\ik (t)|^2 + |\Pi_\ik (t)|^2 + |W_\ik (t)|^2 \right)  \\
	&- \smn \, \left[a(\G_{i-1,k} - 2\G_\ik + \G_{i+1,k}) \G_\ik + a (\G_{i, \,k-1} - 2\G_\ik + \G_{i, \,k+1})\G_\ik \right] \\
	&- \smn \, \left[a(\Pi_{i-1, k} - 2\Pi_\ik + \Pi_{i-1, k}) \Pi_\ik + a (\Pi_{i, \,k-1} - 2\Pi_\ik + \Pi_{i, \,k+1})\Pi_\ik \right] \\
	= &\, \smn \,\left[(f(x_\ik) - f(x_{i-1,k}) - b V_\ik + p(u_\ik - u_{i-1,k})\right] \G_\ik \\
	&+ \smn \,\left[ f(x_\ik) - f(x_{i, k-1}) - b W_\ik + p(u_\ik - u_{i, k-1}) \right] \Pi_\ik \\
	&+ \smn \,\left[(c\, \G_\ik - \gd V_\ik) V_\ik + (c\, \Pi_\ik - \gd W_\ik)W_\ik \right]  \\
	= &\, \smn \, \left[ (f(x_\ik) - f(x_{i-1,k})) \G_\ik + (f(x_\ik) - f(x_{i, k-1}) \Pi_\ik \right] \\
	&\,+ \smn \, \left[(c - b)(\G_\ik V_\ik + \Pi_\ik W_\ik)  - \gd (|V_\ik |^2 + |W_\ik |^2) \right] \\
	&\, + \smn \, p \left[ (u_\ik - u_{i-1,k}) \G_\ik + (u_\ik - u_{i, k-1}) \Pi_\ik \right].
	\end{split}
\eeq
By the Assumption \eqref{Asp}, from \eqref{eG} it follows that 
\beq \bl{AG}
	\begin{split}
	&\frac{1}{2} \frac{d}{dt} \smn \left(|\G_\ik (t)|^2 + |V_\ik (t)|^2 + |\Pi_\ik (t)|^2 + |W_\ik (t)|^2 \right)  \\
	&- \smn \, \left[a(\G_{i-1,k} - 2\G_\ik + \G_{i+1,k}) \G_\ik + a (\G_{i, \,k-1} - 2\G_\ik + \G_{i, \,k+1})\G_\ik \right]. \\
	&- \smn \, \left[a(\Pi_{i-1, k} - 2\Pi_\ik + \Pi_{i-1, k}) \Pi_\ik + a (\Pi_{i, \,k-1} - 2\Pi_\ik + \Pi_{i, \,k+1})\Pi_\ik \right]. \\
	\leq &\, \smn \, \left[\ga (|\G_\ik |^2 + |\Pi_\ik |^2) + |c - b| (\G_\ik V_\ik + \Pi_\ik W_\ik)  - \gd (|V_\ik |^2 + |W_\ik |^2) \right]   \\
	&\, + \smn \, p \left[ (u_\ik - u_{i-1,k}) \G_\ik + (u_\ik - u_{i, k-1}) \Pi_\ik \right],   \quad  t \in [0, \infty),
	\end{split}
\eeq
because for any $(i, k) \in \zmn$, there are $0 \leq \xi_\ik \leq 1$ and $0 \leq \eta_\ik \leq 1$ such that
\begin{align*}
	&(f(x_\ik) - f(x_{i-1,k})) \G_\ik + (f(x_\ik) - f(x_{i, k-1})) \Pi_\ik  \\[3pt]
	= &\, f^{\,\prime} (\xi_\ik x_\ik + (1-\xi_\ik) x_{i-1, k}) \G^2_\ik + f^{\,\prime} (\eta_\ik x_\ik + (1-\eta_\ik) x_{i, k-1}) \Pi^2_\ik \leq \ga (|\G_\ik |^2 + |\Pi_\ik |^2).
\end{align*}
Then the further treatmant will go through the following steps.

Step 1. The two sums without the coefficient $a > 0$ on the left-hand side of \eqref{AG} can be expressed as
\beq \bl{rcLp}
	\begin{split}
	& - \smn \, \left[(\G_{i-1,k} - 2\G_\ik + \G_{i+1,k}) \G_\ik + (\G_{i, \,k-1} - 2\G_\ik + \G_{i, \,k+1})\G_\ik \right] \\
	& - \smn \, \left[(\Pi_{i-1, k} - 2\Pi_\ik + \Pi_{i+1, k}) \Pi_\ik - (\Pi_{i, \,k-1} + 2\Pi_\ik + \Pi_{i, \,k+1})\Pi_\ik \right] \\[2pt]
	= &\, - \smn \, \left[(\G_{i+1, k} - \G_\ik)\G_\ik - (\G_\ik - \G_{i-1, k})\G_\ik \right]  \\
	&- \smn \, \left[(\G_{i, k+1} - \G_\ik)\G_\ik - (\G_\ik - \G_{i, k-1})\G_\ik \right]  \\
	&- \smn \, \left[(\Pi_{i+1, k} - \Pi_\ik)\Pi_\ik - (\Pi_\ik - \Pi_{i-1, k})\Pi_\ik \right] \\
	&- \smn \, \left[(\Pi_{i, k+1} - \Pi_\ik)\Pi_\ik - (\Pi_\ik - \Pi_{i, k-1})\Pi_\ik \right].
	\end{split}
\eeq
The four sums on the right-hand side of \eqref{rcLp} can be treated by using the periodic boundary condition \eqref{pBC}. Among them the first sum is
\beq \bl{Gi}
	\begin{split}
	& - \smn \, \left[(\G_{i+1, k} - \G_\ik)\G_\ik - (\G_\ik - \G_{i-1, k})\G_\ik \right]  \\[4pt]
	= &\, - \sum_{k=1}^n \left(\sum_{i=1}^{m-1}\, (\G_{i+1, k} - \G_\ik ) \G_\ik - \sum_{i=2}^m (\G_\ik- \G_{i-1, k})\G_\ik \right)  \\[3pt]
	&\, - \sum_{k=1}^n \left((\G_{m+1, k} - \G_{m, k})\G_{m, k} - (\G_{1, k} - \G_{0, k})\G_{1, k} \right) \\
	= &\,\sum_{k=1}^n \left(\sum_{i=2}^{m}\, (\G_\ik - \G_{i-1, k})^2 + (\G_{1, k}^2 + \G_{m, k}^2) - (\G_{m+1, k}\, \G_{m, k} + \G_{0, k}\, \G_{1,k})\right)   \\[3pt]
	= &\, \sum_{k=1}^n \left( \sum_{i=2}^{m}\, (\G_\ik - \G_{i-1, k})^2 + (\G_{1, k}^2 + \G_{0, k}^2) - 2 \G_{1, k}\, \G_{0,k} \right)  \\[3pt]
	= &\, \sum_{k=1}^n \left(\sum_{i=2}^{m}\, (\G_\ik - \G_{i-1, k})^2 + (\G_{1, k} - \G_{0, k})^2 \right) = \sum_{k=1}^n \sum_{i=1}^{m} \, (\G_\ik- \G_{i-1, k})^2 \geq 0.
	\end{split}
\eeq          
Similarly the three other sums on the right-hand side of \eqref{rcLp} can be treated and result in the following inequalities,
\beq \bl{Gk}
	\begin{split}
	&- \smn \, \left[(\G_{i, k+1} - \G_\ik)\G_\ik - (\G_\ik - \G_{i, k-1})\G_\ik \right] \, = \sum_{i=1}^m \sum_{k=1}^{n} \, (\G_\ik- \G_{i, k-1})^2 \geq 0    \\[4pt]
	&- \smn \, \left[(\Pi_{i+1, k} - \Pi_\ik)\Pi_\ik - (\Pi_\ik - \Pi_{i-1, k})\Pi_\ik \right] = \sum_{k=1}^n \sum_{i=1}^{m} \, (\Pi_\ik - \Pi_{i-1, k})^2 \geq 0   \\[4pt]
	&- \smn \, \left[(\Pi_{i, k+1} - \Pi_\ik)\Pi_\ik - (\Pi_\ik - \Pi_{i, k-1})\Pi_\ik \right] =  \sum_{i=1}^m \sum_{k=1}^{n} \, (\Pi_\ik - \Pi_{i, k-1})^2 \geq 0.    
	\end{split}
\eeq
Substitute the inequalities \eqref{Gi} and \eqref{Gk} into \eqref{rcLp}. Then the nonnegativity of \eqref{rcLp} shows that \eqref{AG} implies

\beq  \bl{GPVW}
	\begin{split}
	&\frac{1}{2} \frac{d}{dt} \smn \left(|\G_\ik (t)|^2 + |V_\ik (t)|^2 + |\Pi_\ik (t)|^2 + |W_\ik (t)|^2 \right) + \smn \gd (|V_\ik |^2 + |W_\ik |^2) \\
	\leq &\, \smn \, \left[\ga (|\G_\ik |^2 + |\Pi_\ik |^2) + |c - b| (\G_\ik V_\ik + \Pi_\ik W_\ik) \right]   \\
	&\, + \smn \, p \left[ (u_\ik - u_{i-1,k}) \G_\ik + (u_\ik - u_{i, k-1}) \Pi_\ik \right]   \\
	\leq &\, \smn \, \left[\ga (|\G_\ik |^2 + |\Pi_\ik |^2) + \frac{1}{2} |c - b| (|\G_\ik^2 + |V_\ik |^2 + |\Pi_\ik |^2 + |W_\ik |^2)  \right]   \\
	&\, + \smn \, p \left[ (u_\ik - u_{i-1,k}) \G_\ik + (u_\ik - u_{i, k-1}) \Pi_\ik \right],   \quad  t \in [0, \infty),
	\end{split}
\eeq

Step 2. The boundary feedback \eqref{bfc1}-\eqref{bfc2} and \eqref{pbc} infer that
\beq \bl{Gu}
	\begin{split}
	 & \smn \, p(u_\ik - u_{i-1, k})\G_\ik = \sum_{k=1}^n \left[\sum_{i=1}^m \, p(u_\ik - u_{i-1, k})(x_\ik - x_{i-1, k})\right]  \\
	 = &\, p \sum_{k=1}^n \left[(u_{1, k} - u_{0, k})(x_{1, k} - x_{0, k}) + (u_{2, k} - u_{1, k})(x_{2, k} - x_{1, k})\right]  \\
	 + &\, p \sum_{k=1}^n (u_{m, k} - u_{m-1, k})(x_{m,k} - x_{m-1,k})  \\
	 = &\, p \sum_{k=1}^n \left[(u_{1, k} - u_{0, k})(x_{1, k} - x_{0, k}) - u_{1, k} (x_{2, k} - x_{1, k}) + u_{m, k} (x_{m,k} - x_{m-1,k}) \right] \\
	 = &\, p \sum_{k=1}^n 2(x_{m,k} - x_{1, k})(x_{1,k} - x_{m,k})    \\
	 - &\, p \sum_{k=1}^n \left[(x_{m,k} - x_{1,k})(x_{2,k} - x_{1,k}) - (x_{1,k} - x_{m,k}) (x_{m,k} - x_{m-1,k})\right] \; (\text{by} \, \eqref{pbc}) \\
	 = &\, p \sum_{k=1}^n \left[- 3(x_{m,k} - x_{1,k})^2 + (x_{m,k} - x_{1,k})(x_{m-1,k} - x_{2,k})\right]   \\
	 \leq &\, p \sum_{k=1}^n \left[- 2(x_{m,k} - x_{1,k})^2 + (x_{m-1, k} - x_{2, k})^2 \right].
	 \end{split}
\eeq
Similarly we can deduce that
\beq  \bl{Pu}
	\begin{split}
	 & \smn \, p(u_\ik - u_{i, k-1}) \Pi_\ik = \sum_{i=1}^m \left[\sum_{k=1}^n \, p(u_\ik - u_{i, k-1})(x_\ik - x_{i, k-1})\right]  \\
	 \leq &\, p \sum_{i=1}^m \left[- 2(x_{i, n} - x_{i, 1})^2 + (x_{i, n-1} - x_{i, 2})^2 \right].
	 \end{split}
\eeq
Substitute \eqref{Gu} and \eqref{Pu} into \eqref{GPVW}. Then we come up with the following differential inequality
\beq  \bl{TK}
	\begin{split}
	&\frac{d}{dt} \smn \left(|\G_\ik (t)|^2 + |\Pi_\ik (t)|^2 + |V_\ik (t)|^2 + |W_\ik (t)|^2 \right) \\[2pt]
	&\, + 2 \gd \smn (|\G_\ik (t)|^2 + |\Pi_\ik (t)|^2 + |V_\ik |^2 + |W_\ik |^2) \\
	\leq &\, \smn \, 2 \left[ (\gd + \ga) (|\G_\ik |^2 + |\Pi_\ik |^2) + |c - b| (|\G_\ik^2 + |V_\ik |^2 + |\Pi_\ik |^2 + |W_\ik |^2)  \right]   \\
	&\, + \smn \, 2p \left[ (u_\ik - u_{i-1,k}) \G_\ik + (u_\ik - u_{i, k-1}) \Pi_\ik \right],  \\
	\leq &\, \smn \, 2 \left[ (\gd + \ga) (|\G_\ik |^2 + |\Pi_\ik |^2) + |c - b| (|\G_\ik^2 + |V_\ik |^2 + |\Pi_\ik |^2 + |W_\ik |^2)  \right]   \\
	&\, - 4p \left[\sum_{k=1}^n (x_{m,k} - x_{1,k})^2 + \sum_{i=1}^m (x_{i, n} - x_{i, 1})^2\right]  \\
	&\, + 2p \left[\sum_{k=1}^n (x_{m-1,k} - x_{2,k})^2 + \sum_{i=1}^m (x_{i, n-1} - x_{i, 2})^2\right], \quad  t \in [0, \infty).
	\end{split}
\eeq

Step 3. Note that \eqref{Q}-\eqref{lsp} in Theorem \ref{Dsp} confirms that for all solutions of \eqref{CN}-\eqref{inc},
$$
	\limsup_{t \to \infty}\, \smn \left(|x_\ik (t, x_\ik^0)|^2 + |y_\ik (t, y_\ik^0|^2\right) < Q
$$
and the bounded ball $B^*$ shown in \eqref{abs} is an absorbing set in the space $H$ for the semiflow of this 2D FitzHugh-Nagumo CNN with the boundary feedback. Therefore, for any given bounded set $B \subset H$ and any initial data $((x_1^0, y_i^0), \cdots , (x_n^0, y_n^0)) \in B$, there is a finite time $T_B \geq 0$ such that 
\beq \bl{bd}
	\begin{split}
	 & \smn \, 2 \left[ (\gd + \ga) (|\G_\ik |^2 + |\Pi_\ik |^2) + |c - b| (|\G_\ik^2 + |V_\ik |^2 + |\Pi_\ik |^2 + |W_\ik |^2)  \right]   \\
	 & + 2p \left[\sum_{k=1}^n (x_{m-1,k} - x_{2,k})^2 + \sum_{i=1}^m (x_{i, n-1} - x_{i, 2})^2\right]    \\[3pt]
	 < &\, 4 \left(\gd + \ga + 2|c - b| \right) Q + 4p\,Q = 4 \left(\gd + \ga + 2|c - b|  + p\right) Q, \quad \text{for} \;\, t \geq T_B.
	\end{split}
\eeq
Here we used \eqref{DF} which implies that for $t \geq T_B$,
$$
	\smn (|\G_\ik |^2 + |\Pi_\ik |^2) < 2Q , \quad  \smn (|\G_\ik^2 + |V_\ik |^2 + |\Pi_\ik |^2 + |W_\ik |^2) < 4Q,
$$
and
$$
	\sum_{k=1}^n (x_{m-1,k} - x_{2,k})^2 + \sum_{i=1}^m (x_{i, n-1} - x_{i, 2})^2 < 2Q.
$$
Combining \eqref{TK} and \eqref{bd}, we have shown that
\beq \bl{Mq}
	\begin{split}
	&\frac{d}{dt} \smn \left(|\G_\ik (t)|^2 + |\Pi_\ik (t)|^2 + |V_\ik (t)|^2 + |W_\ik (t)|^2 \right) \\
	&\, + 2 \gd \smn (|\G_\ik (t)|^2 + |\Pi_\ik (t)|^2 + |V_\ik (t)|^2 + |W_\ik (t) |^2) \\
	&\, + 4p \left[\sum_{k=1}^n (x_{m,k}(t) - x_{1,k}(t))^2 + \sum_{i=1}^m (x_{i, n}(t) - x_{i, 1}(t))^2\right]  \\[3pt]
	< &\, 4\left(\gd + \ga + 2|c - b|  + p\right) Q, \quad  \text{for} \;\,  t \geq T_B.
	\end{split}
\eeq

	For any given initial state $(x^0, y^0) = ((x_1^0, y_1^0), \cdots, (x_n^0, y_n^0)) \in H$ as a set of single point, there exists a finite time $T_{(x^0, \, y^0)}  > 0$ such that the differential inequality \eqref{Mq} holds for $t \geq T_{(x^0, \, y^0)}$ as well as 
$$
	\smn \left(|x_\ik (t, x_\ik^0) |^2 + |y_\ik (t, y_\ik^0)|^2 \right) < Q, \quad \text{for} \;\; t \geq T_{(x^0, \, y^0)}.
$$
Under the threshold condition \eqref{SC} of this theorem, for $t \geq T_{(x^0, \, y^0)}$,
\beq \bl{Thrs}
	 4p \left[\sum_{k=1}^n (x_{m,k}(t) - x_{1,k}(t))^2 + \sum_{i=1}^m (x_{i, n}(t) - x_{i, 1}(t))^2 \right] > 4\left(\gd + \ga + 2|c - b|  + p\right) Q.
\eeq
It follows from \eqref{Mq} and \eqref{Thrs} that
\beq \bl{Gwq}
\begin{split}
	&\frac{d}{dt} \smn \left(|\G_\ik (t)|^2 + |\Pi_\ik (t)|^2 + |V_\ik (t)|^2 + |W_\ik (t)|^2 \right) \\
	+ \, 2 \gd &\, \smn (|\G_\ik (t)|^2 + |\Pi_\ik (t)|^2 + |V_\ik (t) |^2 + |W_\ik (t)|^2) < 0, \;\; t \geq T_{(x^0, \, y^0)}.
	\end{split}
\eeq
Finally, the Gronwall inequality applied to \eqref{Gwq} shows that
\beq \bl{Syn}
	\begin{split}
	&\smn (|\G_\ik (t)|^2 + |V_\ik (t)|^2) + \smn (|\Pi_\ik (t)|^2 + |W_\ik (t)|^2)  \\
	\leq & e^{- 2\gd [t - T_{(x^0, \, y^0)}]} \sum_{i=1}^n [ |\G_\ik (T_{(x^0, \, y^0)})|^2 + |\Pi_\ik (T_{(x^0, \, y^0)})|^2 + |V_\ik (T_{(x^0, \, y^0)})|^2 + |W_\ik (T_{(x^0, \, y^0)})|^2]  \\[3pt]
	\leq &\, 4e^{- 2\gd [t - T_{(x^0, \, y^0)}]}\,Q \to 0,  \quad \text{as} \;\, t \to \infty.
	\end{split}
\eeq
Thus it is proved that for all solutions of the problem \eqref{CN}-\eqref{Asp} for this 2D FitzHugh-Nagumo CNN with the boundary feedback, the following convergence holds at a uniform exponential rate,
\beq \bl{gik}
	\begin{split}
	&\lim_{t \to \infty} \smn \left(|(x_\ik (t) - x_{i-1,k}(t) |^2 + |(y_\ik (t) - y_{i-1,k}(t) |^2 \right) = 0;  \\
	&\lim_{t \to \infty} \smn \left(|(x_\ik (t) - x_{i,k-1}(t) |^2 + |(y_\ik (t) - y_{i,k-1}(t)|^2 \right) = 0.
	\end{split}
\eeq
The convergence in \eqref{gik} shows that this 2D FitzHugh-Nagumo CNN with the boundary feedback is row synchronized and column synchronized. Therefore, it is uniformly and exponentially synchronized. The proof is completed.
\end{proof}

\textbf{Conclusions}. We summarize the new contributions in this paper.  

1. We propose a new mathematical model of 2D cellular neural networks, whose cell template is the 2D lattice FitzHugh-Nagumo equations with boundary feedback control  \eqref{CN}-\eqref{bfc2}. It features the synaptic coupling in terms of the 2D discrete Laplacian operator and the more meaningful and implementable boundary feedback instead of the pinning control or space-clamped feedback placed in all the interior cell nodes of the network. The boundary feedback is computationally better than the mean-field feedback structures studied in synchronization of neural networks. 

2. For this new model of cellular neural networks, we tackle the global dynamics of the solutions by the approach of dynamical system analysis instead of algebraic spectral analysis and Lyapunov functional. Through the uniform \emph{a priori} estimates,  we proved the existence of an absorbing set in the state space $H = \ell^2 (\mZ_{mn}, \mR^{2mn}) $, which signifies that the CNN system is dynamically dissipative and paves the way towards proof of the main result on sybchronization. 

3. The main result stated in Theorem \ref{ThM} provides a sufficient condition for the exponential synchronization of the 2D FitzHugh-Nagumo cellular neural networks with boundary feedback. The threshold condition \eqref{SC} on
$$
	 \liminf_{t \to \infty} \left[\sum_{k=1}^n |x_{m,k}(t) - x_{1,k}(t)|^2 + \sum_{i=1}^m |x_{i, n}(t) - x_{i, 1}(t)|^2 \right]
$$ 
is to be satisfied by the boundary gap signals between the pairwise boundary cells of the two-dimensional grid. The threshold in \eqref{SC} with \eqref{Q} is explicitly expressed by the parameter and is adjustable by the feedback coefficient $p$ designed in applications.

4. More importantly, the exponential synchronization result and the new methodology contributed in this paper can be directly generalized (through more tedious steps though) to the 3D and even higher dimensional CNN modeled by the corresponding cell template of lattice FitzHugh-Nagumo equations with boundary feedback or by some other type models such as the lattice Hindmarsh-Rose equations. 

We comment on two related open problems. One is that we can change the orthogonal cellular network intercouplng defined by the neighborhood indices $|\widetilde{i} - i | + |\widetilde{k} - k | = 1$ of the cell $N(i, k)$ shown in \eqref{CN} to a different coupling neighborhood $| \widetilde{i} - i | = |\widetilde{k} - k| = 1$ for the two-dimensional CNN \cite{Chow1, Chua2, S2}. Then the discrete Laplacian operator is to be adapted and we conjecture that the same kind of results can still be achieved through the approach from dissipativity to synchronization. Another challenging problem is whether the synchronization is achievable for the same setting of the 2D FitzHugh-Nagumo CNN with the periodic boundary conditions \eqref{pbc} but the boundary feedback \eqref{bfc1} and \eqref{bfc2} are exclusively restricted to the equations associated with the four corner cells $\{N(i, k): i = 1, m;\, k = 1, n\}$. 

The presented modeling and synchronization of the cellular neural networks with boundary feedback are expected to be useful and effective with potential applications in the front of artificial intelligence.

\bibliographystyle{amsplain}

\end{document}